\documentclass[11 pt]{article} \hoffset=-1cm \topmargin=-1.5cm

\usepackage{latexsym}
\usepackage{amsthm}
\usepackage{amssymb}
\usepackage{amsmath,float}
\usepackage[dvipsnames]{xcolor}
\usepackage{tikz}
\usepackage{parskip}
\usepackage{mathtools}
\usepackage{graphicx}
\usepackage{caption}
\usepackage{subcaption}
\usepackage{xcolor}
\usepackage{url}
\usepackage[normalem]{ulem}

\newtheorem{theorem}{Theorem}[section]

\newtheorem{corollary}[theorem]{Corollary}

\newtheorem{lemma}[theorem]{Lemma}

\newtheorem{problem}{Problem}

\newtheorem{remark}[theorem]{Remark}

\theoremstyle{definition}

\title{On a question of Haemers regarding vectors in the nullspace of Seidel matrices}

\author{S. Akbari$^{{\rm 1,}}$\footnote{E-mail addresses:
		s$\_$akbari@sharif.edu (S. Akbari), cioaba@udel.edu (S.M. Cioab\u{a}), sg.samiragoudarzi@gmail.com (S. Goudarzi), aidin.niaparast@gmail.com (A. Niaparast), artin.taj@gmail.com (A. Tajdini)}, S. M. Cioab\u{a}$^{{\rm 2}}$, S. Goudarzi$^{{\rm 1}}$, A. Niaparast$^{{\rm 1}}$, and A. Tajdini$^{{\rm 1}}$
	\\[2mm]
	${\rm ^{1}}$\small Department of Mathematical Sciences, Sharif University of Technology, Tehran, Iran
	\\[2mm]
	${\rm ^{2}}$\small Department of Mathematical Sciences, University of Delaware, USA			
}

\date{}
\begin{document}
	\maketitle
	\begin{abstract}
		In 2011, Haemers asked the following question:
		\begin{center}
			\textit{If $S$ is the Seidel matrix of a graph of order $n$ and $S$ is singular, does there exist an eigenvector of $S$ corresponding to $0$ which has only $\pm 1$ elements?}
		\end{center}
		
		In this paper, we construct infinite families of graphs which give a negative answer to this question. One of our constructions implies that for every natural number $N$, there exists a graph whose Seidel matrix $S$ is singular such that for any integer vector in the nullspace of $S$, the absolute value of any entry in this vector is more than $N$. We also derive some characteristics of vectors in the nullspace of Seidel matrices, which lead to some necessary conditions for the singularity of Seidel matrices. Finally, we obtain some properties of the graphs which affirm the above question.
	\end{abstract}
	
	\section{Introduction}
	
	Throughout this paper all graphs are simple, undirected, and finite. Let $G$ be a graph. The adjacency matrix of graph $G$ with vertex set $V = \{v_1, \ldots. v_n\}$ is denoted by $A=[a_{ij}]_{1\leq i,j\leq n}$, where $a_{ij}=1$ if $v_{i}$ and $v_{j}$ are adjacent, and $a_{ij}=0$ otherwise. Denote by $J_n$ the all-ones $n \times n$ matrix and by $I_n$ the identity matrix of order $n$. The \textit{Seidel matrix} of $G$, denoted by $S$, is defined by $S=J_n-I_n-2A$. There are close connections between Seidel matrices, equiangular lines and two-graphs \cite{balla,mod8_paper,seidel_survey}.
	
	It is known that $rank \,S \ge n - 1$, see \cite[Lemma 3.3]{seidel_rank} or \cite[p.6]{haemers_n}. Assume $rank \, S = n - 1$. Because all elements of the Seidel matrix are integers, there exists an eigenvector corresponding to $0$ with rational entries. Therefore, there exists an integral vector in the nullspace of $S$ such that the greatest common divisor of its entries is 1, and since $null\, S = 1$, this vector is unique up to multiplying by $-1$. We choose that vector whose first non-zero entry is positive, and denote it by $\phi(G)$ or $\phi(S)$ or simply $\phi$. The following problem on the nullspace of singular Seidel matrices was proposed by Haemers in 2011 (see \cite[Problem 3.36]{[2]}).
	\newline
	\begin{problem} \label{problem:hae}	If $S$ is the Seidel matrix of a graph of order $n$ and $rank\,S = n-1$, does there exist an eigenvector of $S$ corresponding to $0$ which has only $\pm 1$ elements? \end{problem}
	This problem is equivalent to finding out whether $\phi(S)$ only has $\pm 1$ entries or not.
	
	In this paper, we construct infinite families of graphs which give a negative answer to Problem \ref{problem:hae}. Moreover, we show that for every natural number $N$, there exists a graph whose Seidel matrix is singular, and the absolute value of every entry of $\phi$ is more than $N$. Furthermore, we investigate some properties of $\phi$, and using these properties, we obtain some necessary conditions for singularity of the Seidel matrix. Finally, we study the graphs having singular Seidel matrices with $\phi \in \{\pm 1\}^n$, and obtain some properties of such graphs.
	
	\section{Preliminaries} \label{section:pre}
	
	In this section, we introduce some notations, concepts, and results which are needed for the next sections.
	
	By $a \overset{m}{\equiv} b$, we mean $a$ and $b$ are congruent modulo $m$.
	Let $G$ be a graph with vertex set $V(G)$ and edge set $E(G)$. By \textit{order} and \textit{size} of $G$, we mean $|V(G)|$ and $|E(G)|$, respectively. The \textit{open neighborhood} of a vertex $v$ is denoted by $N_G(v)$ or $N(v)$, and is the set of all vertices adjacent to $v$. The \textit{closed neighborhood} of $v$, which is $N(v) \cup \{v\}$, is denoted by $N_G[v]$ or $N[v]$. We denote the \textit{degree} of a vertex $v$ in a graph $G$ by $d_G(v)$ or simply $d(v)$. A vertex $v$ is called \textit{even} if $d(v)$ is even, and \textit{odd} if $d(v)$ is odd. A graph is called \textit{even} if all vertices are even. The complement of $G$ is denoted by $\overline{G}$. We denote the \textit{path} and the \textit{cycle} on $n$ vertices by $P_n$ and $C_n$, respectively.
	
	By \textit{switching} the graph $G$ with respect to a vertex $v$, we mean deleting all the edges from $v$ to its neighbors in $G$ and making $v$ adjacent to its non-neighbors in $G$. If $G'$ is the graph obtained by switching with respect to $v$, then ${N_{G'}(v)=V(G)\backslash N_G[v]}$. Two graphs are called \textit{switching-equivalent} if one can be transformed to the other one by a sequence of switchings. This is an equivalence relation on the set of all graphs, and therefore, it partitions the graphs into equivalence classes called \textit{switching classes}. The switching class of a graph $G$ is denoted by $[G]$.
	
	For $G$ and $A \subseteq V(G)$, switching with respect to the vertices in $A$ in any order gives a unique graph $G'$, which could also be obtained by applying switchings with respect to the vertices in $B = V(G) \backslash A$ in any order. In either case, $G'$ is obtained from $G$ by removing all edges between $A$ and $B$, and adding all non-edges between $A$ and $B$ as edges. Therefore, applying some switchings to a graph $G$ is equivalent to partitioning its vertex set into two parts and then complementing the edges between them. When doing so, the parity of the degree of a vertex changes if and only if the number of neighbors this vertex loses does not have the same parity as the number of neighbors it gains, which only happens if the part not including this vertex has an odd number of vertices. In a graph of odd order, there would be one part with odd and one with even cardinality; hence the parity of degrees of all the vertices in the part with even size changes, and the parity of degrees of the remaining vertices remains unchanged.
	\newline
	\begin{remark}\label{remark:evenswitch}
		\normalfont
		For any graph of odd order, switching with respect to the odd vertices transforms it into an even graph. Seidel \cite[Theorem 3.5]{[6]} proved that if $G$ is a graph of odd order, then $G$ contains a unique even graph in its switching class (also see \cite[Theorem 3.17]{hage_thesis}).\\
	\end{remark}
	
	\begin{remark}\label{remark:spectrum}
		\normalfont Let $G$ be a graph, and let $\lambda_1, \ldots, \lambda_n$ be the eigenvalues of $S$ with $\alpha_1, \ldots, \alpha_n$ as their corresponding eigenvectors. If $S'$ is the Seidel matrix of $G$ after applying a switching to vertex $v_k$, then $S' = R_kSR_k$, where $R_k$ is a diagonal matrix obtained by negating the $k^{th}$ main diagonal entry of the identity matrix. Since $R_k = R_k^{-1} $, we have $S'R_k=R_kS$, and therefore, $S'R_k\alpha_i = R_kS\alpha_i=\lambda_iR_k\alpha_i$. Hence, $\lambda_1, \ldots, \lambda_n$ and $\alpha'_1, \ldots, \alpha'_n$ are the eigenvalues and the corresponding eigenvectors of $S'$, where $\alpha'_i = R_k \alpha_i$ is obtained by negating the $k^{th}$ entry of $\alpha_i$. Therefore, applying switchings does not change the spectrum of the Seidel matrix, and only negates the entries of eigenvectors corresponding to the switched vertices. Thus, all switching-equivalent graphs have the same Seidel spectrum.
	\end{remark}
	
	\section{Some Properties of $\phi(S)$ and Necessary Conditions for Singularity of $S$} %some
	\label{section:inc}
	In this section, we study the properties of $\phi(S)$ and obtain some necessary conditions for the singularity of the Seidel matrix.
	
	Let $S$ be a singular Seidel matrix. Because $S \phi = 0$, we deduce that for every $1\leq i\leq n$, the following holds:
	\begin{equation}\label{equation:combin}
	\sum_{j: v_j \in N(v_i)} \phi_j - \sum_{\ell: v_{\ell} \notin N[v_i]} \phi_{\ell} = 0.
	\end{equation}
	We use equation \eqref{equation:combin} to obtain some properties of $\phi$.
	\newline
	\begin{lemma} \label{lemma:odd}
		Every entry of $\phi$ is odd.
	\end{lemma}
	\begin{proof}
		By equation ($\ref{equation:combin}$), for $i = 1 ,\ldots, n$, the multiset $\{\, \phi_j \,|\, j \neq i \,\}$  can be partitioned into two multisets with equal sum. Hence, for $i = 1 ,\ldots, n$, $\sum_{j\neq i} \phi_j$ is even. Therefore , all the entries $\phi_1,\ldots,\phi_n$ have the same parity. Since the greatest common divisor of these entries is 1, each of them must be odd.\end{proof}
	
	\begin{theorem} \label{theorem:mod}
		Let $G$ be a graph with vertex set $V = \{v_1,\ldots,v_n\}$, and singular Seidel matrix $S$. For $1 \le i,j \le n$ the following hold: \\
		(i) $ \phi _i - \phi _j  \overset{4}{\equiv}2(d(v_i)-d(v_j))$, \\
		(ii) If $G$ is an even graph, then $\phi _i - \phi _j  \overset{8}{\equiv}2(d(v_i)-d(v_j)) $.
	\end{theorem}
	\begin{proof}
		(i) Let $M=\{\,a\,|\, v_a \notin N[v_i] \cup N[v_j] \, \}$, $\ N_{i\backslash j}=\{\,b\,|\, v_b \in N(v_i) \backslash N[v_j]\, \}$, $\ N_{j\backslash i}=\{\,c\,|\, v_c \in N(v_j) \backslash N[v_i]\, \}$, and $\ N_{ij} =\{\,d\,|\, v_d \in N(v_i) \cap N(v_j)\, \} $. Using equation (\ref{equation:combin}) for vertices $v_i$ and $v_j$, we get that
		\begin{equation*}
		\pm \phi _j + \sum_{b \in N_{i\backslash j}} \phi_b + \sum_{d \in N_{ij}} \phi_d - \sum_{c \in N_{j\backslash i}} \phi_c - \sum_{a \in M} \phi_a = 0,
		\end{equation*}
		and
		\begin{equation*}
		\pm \phi _i + \sum_{c \in N_{j\backslash i}} \phi_c + \sum_{d \in N_{ij}} \phi_d- \sum_{b \in N_{i\backslash j}} \phi_b - \sum_{a \in M} \phi_a = 0.
		\end{equation*}
		If $v_i$ and $v_j$ are adjacent, then both $\phi_i$ and $\phi_j$ have negative signs above. If $v_i$ and $v_j$ are not adjacent, then $\phi_i$ and $\phi_j$ have positive signs in the previous two equations. Either way, by subtracting these equations we get that
		\begin{equation}\label{equation:mod}
		\pm (\phi_i - \phi _j)=2\left(\sum_{b \in N_{i\backslash j}} \phi_b - \sum_{c \in N_{j\backslash i}} \phi_c\right) .
		\end{equation}
		From Lemma \ref{lemma:odd} we know that all the entries of $\phi$ are odd. Therefore, we obtain that $\phi_i-\phi_j\overset{4}{\equiv} \pm (\phi_i-\phi_j)$ and $\sum_{b \in N_{i\backslash j}} \phi_b - \sum_{c \in N_{j\backslash i}} \phi_c  \overset{2}{\equiv} |N_{i\backslash j}|-|N_{j\backslash i}|$. So
		\begin{equation*}
		\phi_i - \phi _j \overset{4}{\equiv} \pm (\phi_i - \phi _j)=2\left(\sum_{b \in N_{i\backslash j}} \phi_b - \sum_{c \in N_{j\backslash i}} \phi_c\right) \overset{4}{\equiv} 2(|N_{i\backslash j}|-|N_{j\backslash i}|).\end{equation*}
		Since $d(v_i)-d(v_j) = |N_{i\backslash j}| - |N_{j\backslash i}|$, we obtain that
		\begin{equation*}
		\phi_i - \phi _j \overset{4}{\equiv} 2(d(v_i)-d(v_j)).
		\end{equation*}
		
		(ii) Since $G$ is even, by part (i), $\phi_1 \overset{4}{\equiv} \cdots \overset{4}{\equiv} \phi_n \overset{4}{\equiv} r$ for some $r\in \{0,1,2,3\}$. Lemma \ref{lemma:odd} implies that $r \in \{1,3\}$. Consequently,	
		\begin{equation*}
		\sum_{b \in N_{i\backslash j}} \phi_b - \sum_{c \in N_{j\backslash i}} \phi_c  \overset{4}{\equiv} (|N_{i\backslash j}|-|N_{j\backslash i}|)r = (d(v_i)-d(v_j))r \overset{4}{\equiv} d(v_i)-d(v_j),
		\end{equation*}
		where the last one holds because $d(v_i)-d(v_j)$ is even and $r \in \{1,3\}$. So by equation (\ref{equation:mod}), we have that
		\begin{equation*}
		\phi_i - \phi _j\overset{8}{\equiv}\pm (\phi_i - \phi _j)= 2\left(\sum_{b \in N_{i\backslash j}} \phi_b - \sum_{c \in N_{j\backslash i}} \phi_c\right)\overset{8}{\equiv}2(d(v_i)-d(v_j)),
		\end{equation*}
		where the first one is true because
		$\phi_i - \phi_j \overset{4}{\equiv} 0$.
		The proof is complete.
	\end{proof}
	\begin{corollary}\label{corollary:4k+1}
		If $G$ is a graph of order $n$ with singular Seidel matrix, then $ n \overset{4}{\equiv} 1$.
	\end{corollary}
	\begin{proof}
		From Lemma \ref{lemma:odd} and equation (\ref{equation:combin}), we conclude that for $i = 1, \ldots, n$, $d_{G}(v_i)-(n-1-d_{G}(v_i)) \overset{2}{\equiv}  0$, hence $n \overset{2}{\equiv} 1$. By Seidel's result mentioned in Remark \ref{remark:evenswitch}, $G$ is switching equivalent to an even graph $G'$. From Remark \ref{remark:spectrum}, we know that switching does not change the Seidel spectrum of a graph, so $S(G')$ is also singular. Let $\phi' = \phi(G')$. By Theorem \ref{theorem:mod}, part (i), we conclude $\phi'_1 \overset{4}{\equiv} \cdots \overset{4}{\equiv} \phi'_n \overset{4}{\equiv} r$, where $r \in \{1,3\}$. Using equation (\ref{equation:combin}), we deduce that for any $i=1,\ldots,n$,
		\begin{equation*}
		d_{G'}(v_i)r  \overset{4}{\equiv}  \sum_{j \in N_{G'}(v_i)} \phi' _j  \overset{4}{\equiv} \sum_{\ell \notin N_{G'}[v_i]} \phi' _\ell\overset{4}{\equiv} (n-1-d_{G'}(v_i))r.
		\end{equation*}
		Since $\gcd(r,4)=1$ and $d_{G'}(v_i)$ is even, we get that $n  \overset{4}{\equiv} 1 + 2d_{G'}(v_i)  \overset{4}{\equiv} 1$.\end{proof}
	Note that Corollary \ref{corollary:4k+1} can also be derived from \cite[Corollary 3.6]{mod8_paper}.
	\newline
	\begin{lemma} \label{lemma:size}
		Let $G$ be an even graph of order $n$ and size $m$. If the Seidel matrix of $G$ is singular, then
		\begin{equation*}
		m \overset{4}{\equiv} \frac{n - 1}{4}.
		\end{equation*}
	\end{lemma}
	\begin{proof}
		By Corollary \ref{corollary:4k+1}, we have $n \overset{4}{\equiv} 1$. Since $G$ is an even graph, by Theorem \ref{theorem:mod}, Part (i) and Lemma \ref{lemma:odd}, we conclude that $\phi_1 \overset{4}{\equiv} \cdots \overset{4}{\equiv} \phi_n \overset{4}{\equiv} r$,
		for some $r\in \{1,3\}$. We define $\phi'$ as follows:
		\[
		\phi' \coloneqq
		\begin{cases}
		\phi &\quad r=1 \\
		-\phi &\quad r=3. \\
		\end{cases}
		\]
		So we have $\phi'_1 \overset{4}{\equiv} \cdots \overset{4}{\equiv} \phi'_n \overset{4}{\equiv} 1 $. Also, since $\phi' \in \{-\phi, \phi\}$, equation (\ref{equation:combin}) holds for $\phi'$. If we add these equations for $i=1,\ldots,n$, we find that,
		\begin{equation*}
		\sum_{i=1}^{n}(\sum_{j: v_j \in N(v_i)} \phi'_j - \sum_{\ell: v_{\ell} \notin N[v_i]} \phi'_{\ell}) = 0.
		\end{equation*}
		Since for $i=1,\ldots,n$, $\phi'_i$ appears $d(v_i)$ times with positive sign and $n-1-d(v_i)$ times with negative sign, we conclude that,
		\begin{equation*}
		\sum_{i=1}^{n}d(v_i)\phi'_i - \sum_{i=1}^{n} (n-1-d(v_i))\phi'_i = 0,
		\end{equation*}
		thus
		\begin{equation}
		\label{equation:d_iphi_i}
		\sum_{i=1}^{n}d(v_i)\phi'_i = \frac{n-1}{2}\sum_{i=1}^{n}\phi'_i.
		\end{equation}
		Because $\phi'_i \overset{4}{\equiv} 1$ and $d(v_i)$ is even for $i=1,\ldots,n$, we deduce that $d(v_i)\phi'_i\overset{8}{\equiv} d(v_i)$.
		Hence, by equation (\ref{equation:d_iphi_i}),
		\begin{equation*}
		\sum_{i=1}^{n}d(v_i) \overset{8}{\equiv} \frac{n-1}{2}\sum_{i=1}^{n}\phi'_i.
		\end{equation*}
		Therefore,
		\begin{equation*}
		m = \frac{1}{2} \sum_{i=1}^{n}d(v_i) \overset{4}{\equiv} \frac{n-1}{4}\sum_{i=1}^{n}\phi'_i \overset{4}{\equiv} \frac{n-1}{4}n  \overset{4}{\equiv} \frac{n-1}{4},  	
		\end{equation*}
		where the last one is true since $n \overset{4}{\equiv} 1$. The proof is complete.
	\end{proof}
	\begin{theorem} \label{theorem:oddsize}
		Let $G$ be a graph of order $n$ and size $m$. If the Seidel matrix of $G$ is singular, then the following holds:
		\begin{equation}
		m + n_{odd} \overset{4}{\equiv} \frac{n - 1}{4},
		\end{equation}
		where $n_{odd}$ is the number of odd vertices of $G$.
	\end{theorem}
	\begin{proof}
		Let $O(G)$ be the set of odd vertices of $G$. Suppose that $G'$ is the graph obtained from $G$ by switching with respect to the vertices in $O(G)$. Let $m'$ be the size of $G'$. By Remark \ref{remark:evenswitch}, $G'$ is an even graph, and by Remark \ref{remark:spectrum}, the Seidel matrix of $G'$ is singular. Lemma \ref{lemma:size} implies that $n \overset{4}{\equiv} 1$ and $m' \overset{4}{\equiv} \frac{n - 1}{4}$. Let $n_{even}$ be the number of even vertices of $G$ and let $e_G(v)$ be the number of even vertices adjacent to $v$ in $G$. Since $G'$ is obtained from $G$ by switching the edges and non-edges of $G$ between $O(G)$ and $V(G)\backslash O(G)$, we have that
		\begin{equation*}
		m' - m = \sum_{v: v \in O(G)} (n_{even} - e_G(v)) - \sum_{v: v \in O(G)} e_G(v) =  \sum_{v: v \in O(G)} (n_{even} - 2e_G(v)).
		\end{equation*}
		Thus,
		\begin{equation*}
		m' = m + n_{odd}n_{even} - 2\sum_{v: v \in O(G)} e_{G}(v) = m + n_{odd}(n - n_{odd}) - 2m_{oe},
		\end{equation*}
		where $m_{oe}$ is the number of edges of $G$ between $O(G)$ and $V(G)\backslash O(G)$. The Handshaking lemma implies that $n_{odd}$ is even. Also, by adding up the degrees of the vertices in $O(G)$ or in its complement, we deduce that $m_{oe}$ is also even. Therefore,
		\begin{equation*}
		m' = m + n_{odd}n - n_{odd}^2 - 2m_{oe} \overset{4}{\equiv} m + n_{odd}n \overset{4}{\equiv} m + n_{odd},
		\end{equation*}
		and the proof is complete.
	\end{proof}
	Now, we have an immediate corollary.
	\newline
	\begin{corollary}
		\label{coro:size}
		Let $G$ be a graph of order $n$ and size $m$. If the Seidel matrix of $G$ is singular, then  $m \overset{2}{\equiv} \frac{n - 1}{4}$.	
	\end{corollary}
	
	It is worth mentioning that Corollary \ref{coro:size} is weaker than Theorem \ref{theorem:oddsize}. For example, for each $k \geq 0$, non-singularity of the Seidel matrix of $P_{16k+1}$ can be deduced from Theorem \ref{theorem:oddsize} but cannot be derived from Corollary \ref{coro:size}. Note that Corollary \ref{coro:size} can also be concluded from \cite[Theorem 3.5]{mod8_paper}.
	
	Corollary \ref{coro:size} implies that if $G$ is a tree with singular Seidel matrix, then $n \overset{8}{\equiv} 1$. By a computer search, we noted that there is no tree of order 9, whose Seidel matrix is singular. Among 48629 non-isomorphic trees of order 17 (see \cite{McKay}), there are 15 trees with singular Seidel matrix. Also, we checked that Problem \ref{problem:hae} has an affirmative answer for only 2 of these trees. These two trees can be seen in Figure \ref{fig:tree}. In every labeling of these trees that $v_1$ is an even vertex, $\phi_i = (-1)^{d(v_i)}$, for $i=1,\ldots,17$.
	Our program, along with the other 13 trees with singular Seidel matrices, are available in GitHub \cite{GitHub}.
	\begin{figure}[h]
		\begin{subfigure}{.5\textwidth}	
			\centering
			\begin{tikzpicture}[x=0.75pt,y=0.75pt,yscale=-1,xscale=1]
			
			\fill[black] (320,50) circle (4);
			
			\draw (324, 50) -- (344, 50);
			\draw (348, 50) -- (368, 70);
			\draw (368, 70) -- (396, 70);
			\draw (396, 70) -- (424, 70);	
			\draw (348, 50) -- (368, 30);
			\draw (368, 30) -- (396, 30);
			\draw (396, 30) -- (424, 30);
			\fill[black] (348, 50) circle (4);	
			\fill[black] (368, 70) circle (4);	
			\fill[black] (396, 70) circle (4);
			\fill[black] (424, 70) circle (4);	
			\fill[black] (368, 30) circle (4);
			\fill[black] (396, 30) circle (4);
			\fill[black] (424, 30) circle (4);	
			
			\draw (316, 50) -- (296, 50);
			\draw (292, 50) -- (272, 70);
			\draw (272, 70) -- (244, 70);
			\draw (292, 50) -- (272, 30);
			\draw (272, 30) -- (244, 30);
			\draw (264, 50) -- (292, 50);
			
			\fill[black] (292, 50) circle (4);
			\fill[black] (264, 50) circle (4);
			\fill[black] (272, 70) circle (4);
			\fill[black] (244, 70) circle (4);
			\fill[black] (272, 30) circle (4);
			\fill[black] (244, 30) circle (4);
			
			\draw (320, 46) -- (320, 26);
			\draw (320, 22) -- (340, 2);
			\draw (340, 2)  -- (368, 2);
			\fill[black] (320, 22) circle (4);
			\fill[black] (340, 2) circle (4);
			\fill[black] (368, 2) circle (4);
			\end{tikzpicture}
			\caption*{$T_1$}
		\end{subfigure}
		\begin{subfigure}{.5\textwidth}
			\centering
			\begin{tikzpicture}[x=0.75pt,y=0.75pt,yscale=-1,xscale=1]
			
			\fill[black] (320,50) circle (4);
			
			\draw (324, 50) -- (344, 50);
			\draw (348, 50) -- (368, 70);
			\draw (368, 70) -- (396, 70);
			\draw (396, 70) -- (424, 70);	
			\draw (348, 50) -- (368, 30);
			\draw (368, 30) -- (396, 30);
			\draw (396, 30) -- (424, 30);
			\fill[black] (348, 50) circle (4);	
			\fill[black] (368, 70) circle (4);	
			\fill[black] (396, 70) circle (4);
			\fill[black] (424, 70) circle (4);	
			\fill[black] (368, 30) circle (4);
			\fill[black] (396, 30) circle (4);
			\fill[black] (424, 30) circle (4);	
			
			\draw (316, 50) -- (292, 50);
			\draw (292, 50) -- (264, 50);
			\draw (264, 50) -- (244, 70);
			\draw (264, 50) -- (244, 30);
			\draw (264, 50) -- (236, 50);
			\draw (236, 50) -- (208, 50);
			
			\fill[black] (292, 50) circle (4);
			\fill[black] (264, 50) circle (4);
			\fill[black] (236, 50) circle (4);
			\fill[black] (208, 50) circle (4);
			\fill[black] (244, 70) circle (4);
			\fill[black] (244, 30) circle (4);
			
			\draw (320, 46) -- (320, 26);
			\draw (320, 22) -- (340, 2);
			\draw (340, 2)  -- (368, 2);
			\fill[black] (320, 22) circle (4);
			\fill[black] (340, 2) circle (4);
			\fill[black] (368, 2) circle (4);
			\end{tikzpicture}
			\caption*{$T_2$}
		\end{subfigure}
		
		\caption{The only two trees of order 17 with singular Seidel matrices, where $\phi \in \{\pm 1\}^{17}$.}
		\label{fig:tree}
	\end{figure}
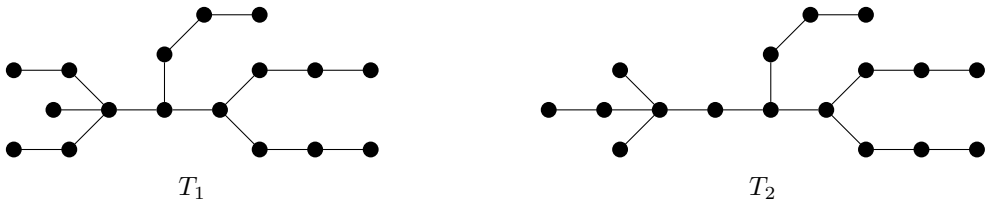
	
	\section{Graph Families with Unbounded Entries of $\phi$ }
	
	Let $G$ be a graph of order $n$ with a singular Seidel matrix. By Corollary \ref{corollary:4k+1}, we know that $n=4k+1$, for some $k\geq 0$. By a computer search, we verified that for every graph of order 5, Problem \ref{problem:hae} has an affirmative answer. However, the following theorem shows that for $k\geq 2$, there is a graph of order $4k+1$, which gives a negative answer to Problem \ref{problem:hae}. The presented family of graphs also shows the unboundedness of entries of $\phi$. Furthermore, in Theorem \ref{theorem:inf}, we show that there is no upper bound for the minimum absolute value of the entries of $\phi$ either. Thus, for every positive integer $N$, there exists a graph whose Seidel matrix is singular, and the absolute value of every entry of $\phi$ is at least $N$.
	\newline
	\begin{theorem}
		\label{theorem:haem}
		For every $k \ge 2$, there exists a graph of order $4k + 1$ with singular Seidel matrix, which gives a negative answer to Problem \ref{problem:hae}.
	\end{theorem}
	\begin{proof}
		For $k \geq 1$, we construct a graph $G_k$ of order $4k + 1$ with singular Seidel matrix for which the maximum entry of $\phi$ is $5^{k - 1}$.
		Hence, $G_k$ gives a negative answer to Problem \ref{problem:hae} for $k\ge2$. The construction is recursive. Let $G_1=C_5$, so $S(G_1)=2J_5-I_5-A(C_5)$ and $\phi(G_1)=j_5$, the all-ones vector of dimension $5$. The structure of $G_{k+1}$ in terms of $G_k$ is represented in Figure \ref{fig:graph1}.
		\begin{figure}[h]
			\centering
			\begin{tikzpicture}[x=0.75pt,y=0.75pt,yscale=-1,xscale=1]
			
			\draw (320,0) ellipse (80 and 20);
			\fill[white] (320, -18) circle(12);
			\draw (320,-18) node   {$G_k$};
			
			\draw (270,5) -- (260, 65);
			\draw (310,5) -- (260, 65);
			\draw (370,5) -- (260, 65);
			
			\draw (270,5) -- (380, 65);
			\draw (310,5) -- (380, 65);
			\draw (370,5) -- (380, 65);
			
			\draw (260, 65) -- (380, 65);
			
			\draw (270,-5) node[scale = 0.7] {$v_{1}$};
			\fill[black] (270,5) circle (5.5);
			\draw (310,-5) node[scale = 0.7] {$v_{2}$};
			\fill[black] (310,5) circle (5.5);
			\fill[black] (332,5) circle (1);
			\fill[black] (340,5) circle (1);
			\fill[black] (348,5) circle (1);
			\draw (370,-5) node[scale = 0.7] {$v_{4k+1}$};;
			\fill[black] (370, 5) circle (5.5);
			
			\draw (260,77) node[scale = 0.7] {$v_{4k + 2}$};
			\fill[black] (260, 65) circle (5.5);
			\draw (300,77) node[scale = 0.7] {$v_{4k + 3}$};
			\fill[black] (300,65) circle (5.5);
			\draw (340,77) node[scale = 0.7] {$v_{4k + 4}$};
			\fill[black] (340,65) circle (5.5);
			\draw (380,77) node[scale = 0.7] {$v_{4k + 5}$};	
			\fill[black] (380, 65) circle (5.5);
			
			\end{tikzpicture}
			\caption{Structure of Graph $G_{k + 1}$}
			\label{fig:graph1}
		\end{figure}
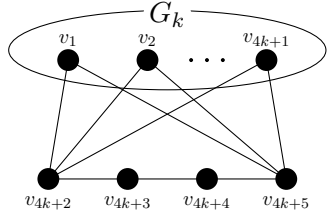	
		Also, the structures of $S(G_{k+1})$ and $\phi(G_{k+1})$ in terms of $S(G_k)$ and $\phi(G_k)$ are as follows:	
		\begin{align*}
		\underbrace{
			\left[
			\begin{array}{c|c}
			\mbox{\normalfont\large\bfseries $S(G_k)$} &
			\begin{matrix}
			-1 & +1 & +1 & -1 \\
			-1 & +1 & +1 & -1 \\
			\vdots & \vdots & \vdots & \vdots \\
			-1 & +1 & +1 & -1
			\end{matrix}
			\\ \hline
			\begin{matrix}
			-1 & -1 & \cdots & -1 \\
			+1 & +1 & \cdots & +1 \\
			+1 & +1 & \cdots & +1 \\
			-1 & -1 & \cdots & -1
			\end{matrix} &
			\begin{matrix}
			0 & -1 & +1 & +1 \\
			-1 & 0 & -1 & +1 \\
			+1 & -1 & 0 & -1 \\
			+1 & +1 & -1 & 0
			\end{matrix}
			\end{array}\right]}_{S(G_{k+1})}
		\underbrace{
			\left[
			\begin{array}{c}
			\vspace{2.85mm}
			\\
			\mbox{\normalfont\large\bfseries $\phi(G_k)$}
			\vspace{2.85mm}
			\\
			\\ \hline
			\begin{matrix}
			c_k \\
			c_k \\
			c_k \\
			c_k \\
			\end{matrix}
			\end{array}\right]}_{\phi(G_{k+1})}
		= 			
		\left[
		\begin{array}{c}
		\begin{matrix}			
		0 \\
		0 \\
		\vspace{0.35mm} \\
		\vdots \\
		\vspace{0.35mm} \\
		0
		\end{matrix}
		\end{array}\right],
		\end{align*}		
		where $c_k$ is the sum of the entries of $\phi(G_k)$. Clearly, the greatest common divisor of entries of
		$[\, \phi(G_k) \, | \, c_k \ c_k \ c_k \ c_k \,]^T$
		is 1 and its first entry is positive; thus, it is consistent with the definition of $\phi$. It can be seen that $c_{k+1} = 5c_{k}$, and since $c_1=5$, we have $c_{k+1} = 5^{k+1}$. Furthermore, it is clear that the maximum entry of $\phi(G_{k+1})$ is $c_{k}$, and the proof is complete.
	\end{proof}
	
	\begin{remark}
		\normalfont In the graphs $G_k$ constructed in the proof of Theorem \ref{theorem:haem}, while the minimum absolute value of the entries of $\phi$ is $1$, the maximum absolute value of its entries tends to infinity as $k$ tends to infinity. \\
	\end{remark}
	
	\begin{theorem}\label{theorem:inf}
		There is no constant upper bound for the minimum absolute value of the entries of $\phi$.
	\end{theorem}
	\begin{proof}
		By induction on $k\geq0$, we construct a graph $H_k$ of order $8k + 5$ with singular Seidel matrix such that the minimum absolute value of the entries of $\phi(H_k)$ is $3^k$, and the sum of its entries is $5 \times 7^k$. Thus, for any natural number $N$, there is $k$ such that the minimum absolute value of entries of the vector $\phi$ of $H_k$ is greater than $N$.
		
		Let $H_0=C_5$. As before, $\phi(H_0) = j_5$, where $j_5$ is the all-ones vector of order $5$, so the minimum absolute value of the entries of $\phi(H_0)$ is $1$, and the sum of its entries is $5$. The structure of graph $H_{k+1}$ in terms of $H_k$ is shown in Figure \ref{fig:graph2}.
		\begin{center}
			\begin{figure}[!h]
				\centering
				\begin{tikzpicture}[x=0.75pt,y=0.75pt,yscale=-1,xscale=1]
				
				\draw (320,0) ellipse (80 and 20);
				\fill[white] (320, -18) circle(12);
				\draw (320,-18) node   {$H_k$};
				
				\draw (270,5) -- (300, 65);
				\draw (310,5) -- (300, 65);
				\draw (370,5) -- (300, 65);
				
				\draw (270,5) -- (340, 65);
				\draw (310,5) -- (340, 65);
				\draw (370,5) -- (340, 65);
				
				\draw (260,110) -- (380,110);
				\draw (300,140) -- (340,140);
				\draw (300,110) -- (300,140);	
				\draw (340,110) -- (340,140);	
				
				\draw (270,-5) node[scale = 0.7] {$v_{1}$};
				\fill[black] (270,5) circle (5.5);
				\draw (310,-5) node[scale = 0.7] {$v_{2}$};
				\fill[black] (310,5) circle (5.5);
				\fill[black] (332,5) circle (1);
				\fill[black] (340,5) circle (1);
				\fill[black] (348,5) circle (1);
				\draw (370,-5) node[scale = 0.7] {$v_{8k+5}$};;
				\fill[black] (370, 5) circle (5.5);
				
				\draw (300,77) node[scale = 0.7] {$v_{8k + 9}$};
				\fill[black] (300,65) circle (5.5);
				\draw (340,77) node[scale = 0.7] {$v_{8k + 10}$};
				\fill[black] (340,65) circle (5.5);
				
				\draw (300,100) node[scale = 0.7] {$v_{8k + 6}$};
				\fill[black] (300,110) circle (5.5);
				\draw (340,100) node[scale = 0.7] {$v_{8k + 13}$};
				\fill[black] (340,110) circle (5.5);
				\draw (260,100) node[scale = 0.7] {$v_{8k + 11}$};
				\fill[black] (260, 110) circle (5.5);
				\draw (380,100) node[scale = 0.7] {$v_{8k + 8}$};	
				\fill[black] (380, 110) circle (5.5);
				\draw (300,152) node[scale = 0.7] {$v_{8k + 12}$};
				\fill[black] (300, 140) circle (5.5);
				\draw (340,152) node[scale = 0.7] {$v_{8k + 7}$};
				\fill[black] (340,140) circle (5.5);
				\end{tikzpicture}
				\caption{Structure of Graph $H_{k + 1}$}
				\label{fig:graph2}
			\end{figure}
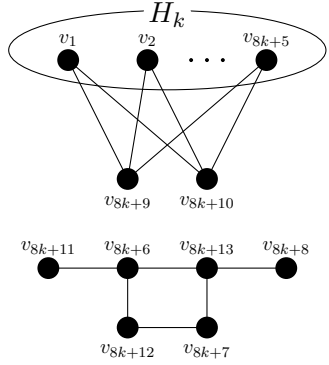
		\end{center}%\pagebreak
		Also, the structures of $S(H_{k+1})$ and $\phi(H_{k+1})$ in terms of $S(H_k)$ and $\phi(H_k)$ are as follows:
		\begin{align*}
		\underbrace{
			\left[
			\begin{array}{c|c}
			\mbox{\normalfont\large\bfseries $S(H_k)$} &
			\begin{matrix}
			+1 & +1 & +1 & -1 & -1 & +1 & +1 & +1\\
			+1 & +1 & +1 & -1 & -1 & +1 & +1 & +1\\
			\vdots & \vdots & \vdots & \vdots &\vdots & \vdots & \vdots & \vdots\\
			+1 & +1 & +1 & -1 & -1 & +1 & +1 & +1
			\end{matrix}
			\\ \hline
			\begin{matrix}
			+1 & +1 & \cdots & +1 \\
			+1 & +1 & \cdots & +1 \\
			+1 & +1 & \cdots & +1 \\
			-1 & -1 & \cdots & -1 \\
			-1 & -1 & \cdots & -1 \\
			+1 & +1 & \cdots & +1 \\
			+1 & +1 & \cdots & +1 \\
			+1 & +1 & \cdots & +1
			\end{matrix} &
			\begin{matrix}
			0 & +1 & +1 & +1 & +1 & -1 & -1 & -1\\
			+1 & 0 & +1 & +1 & +1 & +1 & -1 & -1\\
			+1 & +1 & 0 & +1 & +1 & +1 & +1 & -1\\
			+1 & +1 & +1 & 0 & +1 & +1 & +1 & +1\\
			+1 & +1 & +1 & +1 & 0 & +1 & +1 & +1\\
			-1 & +1 & +1 & +1 & +1 & 0 & +1 & +1\\
			-1 & -1 & +1 & +1 & +1 & +1 & 0 & +1\\
			-1 & -1 & -1 & +1 & +1 & +1 & +1 & 0
			\end{matrix}
			\end{array}\right]}_{S(H_{k+1})}
		\underbrace{
			\left[
			\begin{array}{c}
			\vspace{2.85mm}
			\\
			\mbox{\normalfont\large\bfseries $3\phi(H_k)$}
			\vspace{2.85mm}
			\\
			\\ \hline
			\begin{matrix}
			5c_k \\
			-c_k \\
			-3c_k \\
			c_k \\
			c_k \\
			-3c_k \\
			-c_k \\
			5c_k \\
			\end{matrix}
			\end{array}\right]}_{\phi(H_{k+1})}
		=
		\left[
		\begin{array}{c}
		\begin{matrix}			
		0 \\
		0 \\
		\vspace{4.35mm} \\
		\vdots \\
		\vspace{4.35mm} \\
		0
		\end{matrix}
		\end{array}\right],
		\end{align*}
		where $c_k$ is the sum of the entries of $\phi(H_k)$. Note that $c_k = 5 \times 7^k$, so the greatest common divisor of entries of
		$[\, 3\phi(H_k) \, | \, 5c_k \ {-c_k} \ {-3c_k} \ \ c_k \ \ c_k  \ {-3c_k} \ {-c_k} \ \, {5c_k} \,]^T$
		is $gcd(3,c_k)=1$ and its first entry is positive. Hence, it is consistent with the definition of $\phi$.
		The structure of $\phi(H_{k+1})$ indicates that $c_{k+1} = 7c_k$, thus $c_{k+1} = 5 \times 7^{k+1}$. Furthermore, the minimum absolute value of entries of $\phi(H_{k + 1})$ is equal to the minimum absolute value of $3\phi(H_k)$, which is $3 ^ {k+1}$, and it tends to infinity as $k$ tends to infinity.\end{proof}
	
	\section{Properties of Graphs with $\phi(S)\in \{\pm 1\}^n$} \label{section:5}
	
	In this section, we study the graphs $G$ which give an affirmative answer to the Problem \ref{problem:hae}. This means that $\phi(S)\in \{\pm 1\}^n$. By Corollary \ref{corollary:4k+1}, we know that if $G$ is a graph of order $n$ with singular Seidel matrix, then $n=4k+1$, for some non-negative integer $k$. Hence, we only consider graphs of order $n = 4k + 1$. Our first result gives a characterization of the switching class of such graphs.
	\newline
	\begin{theorem}\label{theorem:switch-reg}
		Let $G$ be a graph of order $n = 4k + 1$. The Seidel matrix $S$ of $G$ is singular and $\phi(S) \in \{\pm 1\}^n$ if and only if $G$ belongs to the switching class of a $2k$-regular graph.\end{theorem}
	\begin{proof}
		Suppose that $G$ is a graph with Seidel matrix $S$ such that $\phi(S) \in \{\pm 1\}^n$. We define $L := \{ \,v_i \, |  \, (\phi(S))_i = -1 \, \}$. By switching $G$ with respect to $L$, we obtain a graph $G'$, with $S'$ and $A'$ as its Seidel matrix and adjacency matrix, respectively. By Remark \ref{remark:spectrum}, $S'$ is singular and $\phi(S')=j_n$, where $j_n$ is the all-ones vector of dimension $n$. Therefore,
		\begin{equation*}
		0 = S'j_n = (J_n-I_n-2A')j_n = (4k + 1)j_n  -j_n - 2A'j_n.
		\end{equation*}
		Thus, $A'j_n=2kj_n$, and so $G'$ is a $2k$-regular graph.
		
		To prove the other direction, suppose that $H$ is a $2k$-regular graph in the switching class of $G$. We have that
		\begin{equation*}
		S(H)j_n = (J_n-I_n-2A(H))j_n = (4k + 1)j_n - j_n - 2 (2k) j_n=0.
		\end{equation*}
		So $j_n$ is an eigenvector of $S(H)$ corresponding to $0$, and therefore $\phi(H) = j_n$. The graph $G$ is obtained from $H$ by some switching. Since switching with respect to any subset of vertices only negates the corresponding entries of its Seidel eigenvectors, it follows that $\phi(S) \in \{\pm 1 \}^n$, which finishes the proof.
	\end{proof}
	
	Theorem \ref{theorem:switch-reg} provides a way for constructing graphs with singular Seidel matrices, which are rare in general. It is shown in \cite{seidel_probab} that the proportion of graphs of order $n$
	with singular Seidel matrices tends to zero as $n$ tends to infinity.	
	
	The next results present some structural properties of graphs with $\phi(S)\in \{\pm 1\}^n$.
	\newline
	\begin{lemma}\label{lemma:leaf}
		Let $G$ be a graph of order $n = 4k + 1$ with singular Seidel matrix $S$, and $\phi(S) \in \{\pm 1\}^n$. If $G$ has a leaf, then $G$ has exactly $2k$ or $2k + 2$ odd vertices.
	\end{lemma}
	\begin{proof} By Theorem \ref{theorem:switch-reg}, there exists a $2k$-regular graph $H$ such that $G$ is obtained from $H$ by some switching. Let $v$ be a leaf of $G$, and $u$ be its neighbor in $G$. There are two cases to consider.
		
		{\em Case 1:} The vertices $u$ and $v$ are adjacent in $H$. In this case, $H$ is transformed into $G$ by switching with respect to $A = N_{H}(v) \backslash \{u\}$ or $B = V \backslash A $. Because $|A| = 2k - 1 \overset{2}{\equiv} 1$ and $|B| = 2k + 2 \overset{2}{\equiv} 0$, complementing the edges between $A$ and $B$ changes the parity of the degree of vertices in $B$. Therefore, in $G$, the vertices in $B$ are odd and the vertices in $A$ are even. So $G$ has $2k + 2$ odd vertices.
		
		{\em Case 2:} The vertices $u$ and $v$ are not adjacent in $H$. In this case, $H$ is transformed into $G$ by switching with respect to $A = N_{H}(v) \cup \{u\}$ or $B = V \backslash A $. In this case, $|A| = 2k + 1 \overset{2}{\equiv} 1$ and $|B| = 2k \overset{2}{\equiv} 0$. Thus, $G$ has $|B| = 2k$ odd vertices. \end{proof}
	\begin{corollary}
		If G is a tree of order $n = 16k + r \ (0 \leq r \leq 15)$ with a singular Seidel matrix, then G has exactly $8k + s$ odd vertices, where $(r, s)$ is either $(1, 0)$ or $(9, 6)$.
	\end{corollary}
	\begin{proof}
		Let $n_{odd}$ be the number of  odd vertices of $G$. By Theorem \ref{theorem:oddsize} and Corollary \ref{corollary:4k+1}, we conclude that $n_{odd} \overset{4}{\equiv} \frac{n - 1}{4}$. By Lemma \ref{lemma:leaf}, $n_{odd}$ is either $\frac{n - 1}{2}$ or $\frac{n + 3}{2}$. Thus, $n_{odd} = \frac{n - 1}{2}$, if $\frac{n - 1}{2} \overset{4}{\equiv} \frac{n - 1}{4}$, which means
		$n \overset{16}{\equiv} 1$, and
		$n_{odd} = \frac{n + 3}{2}$, if $\frac{n + 3}{2} \overset{4}{\equiv} \frac{n - 1}{4}$, which means
		$n \overset{16}{\equiv} 9$.
	\end{proof}
	\begin{theorem}\label{theorem:bounds}
		Let $G$ be a graph of order $n = 4k + 1$. If the Seidel matrix $S$ of $G$ is singular, and $\phi(S) \in \{\pm 1\}^n$, then $3k \leq |E(G)| \leq 8k^2 - k$, and both given bounds are tight.
	\end{theorem}
	\begin{proof} Let $\delta(G)$ denote the minimum degree of $G$. If $\delta(G) \geq 2$, then $|E(G)|\geq 4k+1 > 3k$. Suppose that $\delta(G) < 2$, and let $v$ be a vertex of minimum degree. By Theorem \ref{theorem:switch-reg}, there exists a $2k$-regular graph $H$ such that $G$ is obtained from $H$ by some switching. Let $V$ be the vertex set of $G$ and $H$. There are two possibilities for $d_G(v)$:
		\begin{itemize}
			\item $d_{G}(v) = 0$. Let $\alpha$ be the number of edges in $H$ with endpoints in $N_{H}(v)$, $\beta$ be the number of edges in $H$ with endpoints in $V \backslash N_{H}[v]$, and $\gamma$ be the number of edges in $H$ with one endpoint in $N_{H}(v)$ and the other one in $V \backslash N_{H}[v]$. The $2k$ edges adjacent to $v$ are the only edges of $H$ that are not counted in neither $\alpha$, $\beta$ nor $\gamma$. Thus, we have
			\begin{equation}\label{eq:aaaa}
			\alpha + \beta + \gamma = \dfrac{2k(4k + 1)}{2} - 2k = 4k^2 - k.
			\end{equation}
			Every vertex $u$ in $N_H(v)$ has at most $2k - 1$ neighbors in $V \backslash N_{H}[v]$, so $\gamma \leq (2k)(2k - 1)=4k^2-2k$. Thus, we have
			\begin{equation}
			\label{eq:bbbb}
			4k^2 - 2\gamma \geq 4k^2 - 2(4k^2 - 2k) = 4k - 4k^2.
			\end{equation}
			There exists a set $A \subseteq V$ such that switching with respect to $A$ transforms $H$ into $G$. This switching makes $v$ an isolated vertex. Therefore,
			$A$ is equal to either $N_{H}(v)$ or $V \backslash N_{H}(v)$. By applying this switching, the edges with endpoints in $N_H(v)$ and the edges with endpoints in $V \backslash N_H[v]$ remain as they were, and the edges and non-edges between the two sets are switched. Since $|N_H(v)|=|V\setminus N_H[v]|=2k$, by equations (\ref{eq:aaaa}) and (\ref{eq:bbbb}), we get that
			\begin{equation*}
			|E(G)| = \alpha + \beta + 4k^2 - \gamma = (\alpha + \beta + \gamma)
			+ (4k^2 - 2\gamma) \geq 4k^2 - k + 4k - 4k^2 = 3k.
			\end{equation*}	
			\item $d_{G}(v) = 1$. By Lemma \ref{lemma:leaf}, there are at most $2k + 2$ odd and at least $2k - 1$ even vertices in $G$.
			Since there is no vertex of degree 0, the degree of any even vertex in $G$ is at least $2$. As a result,
			\begin{equation*}
			|E(G)| = \frac{1}{2} \sum_{v \in V}d_{G}(v) \geq \frac{1}{2} (1(2k+2) + 2(2k-1)) = 3k.
			\end{equation*}
		\end{itemize}
		Since $S(\overline{G}) = -S(G)$, we have $\phi(\overline{G}) = \phi(G)$, so by the first inequality for graph $\overline{G}$, $|E(\overline{G})| \geq 3k$, and therefore, $|E(G)| \leq 8k^2 - k$.
		
		In the sequel, we show that the given lower bound and upper bound are tight. Let $k$ be a natural number and $G_k$ be the disjoint union of $k$ copies of $P_4$ and one isolated vertex. Note that the order and size of $G_k$ are $4k + 1$ and $3k$, respectively. By switching with respect to the $2k$ leaves of $G_k$, we obtain a $2k$-regular graph of order $4k + 1$. By Theorem \ref{theorem:switch-reg}, the Seidel matrix of $G_k$ is singular, and $\phi(G_k) \in \{\pm 1\}^n$. The tightness of the upper bound can be shown by considering $\overline{G_k}$.
	\end{proof}
	\begin{remark}
		\label{remark:circ}
		\normalfont For any natural number $k \geq 3$, we construct a graph $H_k$ of order $4k + 1$ and size $3k$, shown in Figure \ref{fig:graph3}, containing one $C_k$, $2k$ leaves, and $k + 1$ isolated vertices. It can be checked that switching the leaves transforms this graph into a $2k$-regular graph, and so by Theorem \ref{theorem:switch-reg}, $\phi(H_k)  \in \{\pm 1\}^n$.
		\begin{figure}[h]
			\centering
			\begin{tikzpicture}[x=0.75pt,y=0.75pt,yscale=-1,xscale=1]
			\draw (290,0) -- (280, 25);	
			\draw (255,0) -- (285,0);
			\draw (355,0) -- (385,0);
			\draw (390, -5) arc (300:240:140);
			\draw (250,0) -- (240, 25);
			\draw (350,0) -- (360, 25);
			\draw (350,0) -- (340, 25);
			\draw (390,0) -- (400, 25);
			\draw (390,0) -- (380, 25);
			\draw (250,0) -- (260, 25);
			\draw (290,0) -- (300, 25);	
			
			\fill[black] (250,0) circle (5.5);
			\fill[black] (290,0) circle (5.5);
			\fill[black] (312,0) circle (1);
			\fill[black] (320,0) circle (1);
			\fill[black] (328,0) circle (1);
			\fill[black] (350,0) circle (5.5);
			\fill[black] (390,0) circle (5.5);
			\fill[black] (240,30) circle (5.5) ;
			\fill[black] (260,30) circle (5.5) ;
			\fill[black] (280,30) circle (5.5) ;
			\fill[black] (300,30) circle (5.5) ;
			\fill[black] (340,30) circle (5.5) ;
			\fill[black] (360,30) circle (5.5) ;
			\fill[black] (380,30) circle (5.5) ;
			\fill[black] (400,30) circle (5.5) ;
			\fill[black] (230,60) circle (5.5);
			\fill[black] (270,60) circle (5.5);
			\fill[black] (312,60) circle (1);
			\fill[black] (320,60) circle (1);
			\fill[black] (328,60) circle (1);
			\fill[black] (370,60) circle (5.5);
			\fill[black] (410,60) circle (5.5);
			
			\end{tikzpicture}
			\caption{Graph $H_k$}
			\label{fig:graph3}
		\end{figure}
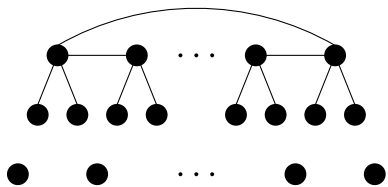
		
	\end{remark}
	The structure given in Remark \ref{remark:circ} indicates that there are some graphs with $\phi \in \{\pm 1\}^n$ and minimum possible size, which are not a forest. Also, it shows that for any $k\geq3$, there is a graph containing $C_k$, which has an affirmative answer to the Problem \ref{problem:hae}.
	
	\textbf{Acknowledgment.} The research of the first author was supported by grant number (G981202) from the Sharif University of Technology. The research of the second author was supported by the grants NSF DMS-1600768, CIF-1815922, and a JSPS Invitational Fellowship for Research in Japan (Short-term S19016). We thank the referee for her/his fruitful comments.

	\bibliographystyle{siam}
	\bibliography{seidel-nullity-bibliography}

\begin{thebibliography}{10}

\bibitem{GitHub}
{\em Jupyter notebook containing trees of order 17 with singular seidel
  matrices}.
\newblock \url{https://github.com/Seidel-matrix-singularity/Trees}, 2020.
\newblock [Online; accessed 9-October-2020].

\bibitem{seidel_rank}
{\sc S.~Akbari, J.~Askari, and K.~C. Das}, {\em Some properties of eigenvalues
  of the {Seidel} matrix}, Linear and Multilinear Algebra,  (2020), pp.~1--12.

\bibitem{balla}
{\sc I.~Balla, F.~Dräxler, P.~Keevash, and B.~Sudakov}, {\em Equiangular lines
  and spherical codes in {Euclidean} space}, Inventiones Mathematicae, 211
  (2018), pp.~179--212.

\bibitem{[2]}
{\sc F.~Belardo, S.~M. Cioabă, J.~Koolen, and J.~Wang}, {\em Open problems in
  the spectral theory of signed graphs}, The Art of Discrete and Applied
  Mathematics, 1 (2018), no. 2, \#P2.10, 23pp.

\bibitem{mod8_paper}
{\sc G.~Greaves, J.~H. Koolen, A.~Munemasa, and F.~Szöllősi}, {\em
  Equiangular lines in {Euclidean} spaces}, Journal of Combinatorial Theory,
  Series A, 138 (2016), pp.~208--235.

\bibitem{haemers_n}
{\sc W.~H. Haemers}, {\em Seidel {Switching} and {Graph} {Energy}}, Match
  Communications in Mathematical and in Computer Chemistry, 68 (2012), no. 3,
  pp.~653 -- 659.

\bibitem{hage_thesis}
{\sc J.~Hage}, {\em Structural {Aspects} of {Switching} {Classes}}, PhD thesis,
  LIACS,  (2001).

\bibitem{McKay}
{\sc B.~D. McKay}, {\em Combinatorial data}.
\newblock \url{http://users.cecs.anu.edu.au/~bdm/data/trees}.

\bibitem{seidel_probab}
{\sc D.~Rizzolo}, {\em Determinants of {Seidel} matrices and a conjecture of
  {Ghorbani}}, Linear Algebra and its Applications, 579 (2019), pp.~51--54.

\bibitem{[6]}
{\sc J.~J. Seidel}, {\em Graphs and two-graphs}, Proceedings 5th Southeastern
  Conference on Combinatorics, Graph Theory and Computing,  (1974),
  pp.~125--143.

\bibitem{seidel_survey}
{\sc J.~J. Seidel}, {\em A survey of two-graphs}, in Geometry and
  {Combinatorics}, D.~G. Corneil and R.~Mathon, eds., Academic Press, Jan.
  1991, pp.~146--176.

\end{thebibliography}
	
\end{document}